\documentclass[10pt]{amsart}

\usepackage{amsmath, amsfonts, amssymb, xypic}
\usepackage{latexsym, verbatim}
\usepackage{graphicx}
\usepackage{color}
\usepackage{url}
\usepackage[table]{xcolor}
\usepackage{hyperref}
\usepackage{amsmath}

\newtheorem{Theorem}{Theorem}[section]
 \newtheorem{Corollary}[Theorem]{Corollary}
 \newtheorem{Lemma}[Theorem]{Lemma}
 \newtheorem{Proposition}[Theorem]{Proposition}

\theoremstyle{definition}
 
 \newtheorem{Remark}[Theorem]{Remark}

\newcommand{\II}{ \mathbb{I}}

\newcommand{\R}{\mathbb{R}}
\newcommand{\CC}{\mathbb{C}}

\newcommand{\cA}{{\mathcal A}}

\newcommand{\La}{\Lambda}

\newcommand{\Ke}{{\rm Ker}\,}

\newcommand{\del}{\partial}
\newcommand{\delb}{{\bar \partial}}
\newcommand{\mub}{{\bar \mu}}

\newcommand{\taub}{{\bar \tau}}

\newcommand{\Dol}{\mathrm{Dol}}

\title{Almost Hermitian identities}

\author[J. Cirici]{Joana Cirici}
\address[J. Cirici]{Department of Mathematics and Computer Science, Universitat de Barcelona\\
Gran Via 585\\
08007 Barcelona }
\email{jcirici@ub.edu}

\author[S. Wilson]{Scott O. Wilson}
  \address[S. Wilson]{Department of Mathematics, Queens College, City University of New York, 65-30 Kissena Blvd., Flushing, NY 11367}
  \email{scott.wilson@qc.cuny.edu}

\thanks{J. Cirici would like to acknowledge partial support from the AEI/FEDER, UE (MTM2016-76453-C2-2-P) and the Serra H\'{u}nter Program.
}

\thanks{S. Wilson acknowledges support provided by a PSC-CUNY Award, jointly funded by The Professional Staff Congress and The City University of New York.}

\begin{document}

\begin{abstract} 
We study the local commutation relation between the Lefschetz operator and the exterior differential on an almost complex manifold with a compatible metric. The identity that we obtain generalizes the backbone of the local K\"{a}hler identities to the setting of almost Hermitian manifolds, allowing for new global results for such manifolds.
\end{abstract}

\maketitle

\section{Introduction}

On a K\"{a}hler manifold $(M,J,\omega)$, the most fundamental local identity is perhaps the commutation relation 
between the exterior differential $d$ and the adjoint $\Lambda$ to the Lefschetz operator,
\begin{equation}\label{theequation}
 [\La, d]= \star\, \mathbb{I}^{-1} \,d \, \mathbb{I}\,  \star,
\end{equation}
where $\star$ denotes the Hodge star operator and $\mathbb{I}$ denotes the extension of $J$ to all forms.

This identity, due to A. Weil \cite{Weil}, strongly depends on the K\"{a}hler condition, $d\omega=0$, and in fact is true when removing the integrability condition $N_J\equiv 0$. So, it is valid for almost K\"{a}hler and also symplectic manifolds as well \cite{BaTo,TaToPreprint,AK}.
On the other hand, there is also a generalization of the K\"{a}hler identities in the Hermitian setting (see \cite{De,SWH}), which strongly uses integrability.

When the manifold is only almost Hermitian, then the above local identity does not hold in general, as noticed implicitly in \cite{Osh}. The purpose of this short note is to show precisely how the above K\"{a}hler identity (\ref{theequation}) becomes modified when the form $\omega$ is not closed. 

The main result is given in Theorem \ref{AHidentity} below, which has several applications including the uniqueness of Dirichlet problem
\[
\del \delb u = g \quad \quad \textrm{with} \quad u |_{\del \Omega} = \phi,
\]
on any compact domain $\Omega$ in an almost complex manifold. This in turn implies that the Dolbeault cohomolgy introduced in \cite{CWDol}, for all almost complex manifolds, satisfies $H_{\Dol}^{0,0}(M)\cong \CC$ for a compact connected almost complex manifold.

Another application of the almost Hermitian identities of Theorem \ref{AHidentity}  appears in forthcoming work by Feehan and Leness \cite{FL}. There the fundamental relation of Proposition \ref{0qforms} is used to show that the moduli spaces of unitary anti-self-dual connections over any almost Hermitian 4-manifold is almost Hermitian, whenever the Nijenhuis tensor has sufficiently small $C^0$-norm. This generalizes a well known result for K\"ahler manifolds that was exploited in Donaldson's work in the 1980's, and is expected to have consequences for the topology of almost complex 4-manifolds which are of so-called Seiberg-Witten simple type.

When $M$ is compact, local identities lead to consequences in cohomology, often governed by geometric-topological inequalities. Indeed, the exterior differential inherits a bidegree decomposition into four components $d=\mub+\del+\del+\mu$ and the Hermitian metric allows one to consider the Laplacian operators associated to each of these components. In the compact case, the numbers 
\[\ell^{p,q}:=\dim \Ke(\Delta_\delb+\Delta_\mu)|_{(p,q)}\]
given by the kernel of $\Delta_\delb+\Delta_\mu$ in bidegree $(p,q)$
are finite by elliptic operator theory.
When $J$ is integrable (and so $M$ is a complex manifold) the operator $\Delta_\mu$ vanishes and these are just the Hodge numbers $\ell^{p,q}=h^{p,q}$. In this case, the Hodge-to-de Rham spectral sequence gives inequalities 
\[\sum_{p+q=k}\ell^{p,q}\geq b^k,\]
where $b^k$ denotes the $k$-th Betti number. On the other hand, as shown in \cite{AK},
one main consequence of the local identity (\ref{theequation}) in the almost K\"{a}hler case $d\omega=0$ is the converse inequality 
\[\sum_{p+q=k}\ell^{p,q}\leq b^k.\]
Of course, in the integrable K\"{a}hler case both inequalities are true and so one recovers the well-known consequence of the Hogde decomposition 
\[\sum_{p+q=k}\ell^{p,q}=b^k.\]
The local identities of \cite{De,SWH} for complex non-K\"{a}hler manifolds include 
other algebra terms which lead to further Laplacian operators, leading also to various inequalities relating the geometry with the topology of the manifold. 

With this note, we aim to further understand the origin of these inequalities 
by means of the correct version of (\ref{theequation}) for almost Hermitian manifolds
for which, a priori, the only geometric-topological inequality in the compact case is given by 
\[\sum_{p+q=k}\dim \Ke(\Delta_\mub+\Delta_\delb+\Delta_\del+\Delta_\mu)|_{(p,q)}\leq b^k.\] 
\medskip
\noindent
\textbf{Acknowledgments. }The authors would like to thank Paul Feehan for encouraging us develop some previous notes into the present paper.

\section{Preliminaries}

Let $(\cA,d)$ denote the complex valued differential forms of an almost complex manifold $(M,J)$. For any Hermitian metric, define the associated Hodge-star operator \[\star: \cA^{p,q}_x \to \cA^{n-q,n-p}_x\,\,\text{ by }\,\,
\omega \wedge \star \bar \eta = \langle \omega , \eta \rangle \textrm{vol},
\] 
where $\omega$ is the fundamental $(1,1)$-form, and $\textrm{vol} = \frac{1}{n!} \omega^n \in \cA^{n,n}$ is the volume form determined by the Hermitian metric. 
Note $\star^2 = (-1)^k$ on $\cA^k$.

Define $d^* = - \star d \star$, so that $d^*\star  = (-1)^{k+1} \star d$ on $\cA^k$. Similarly, consider the bidegree decomposition of the exterior differential 
\[d= \mub+\delb+\del+\mu,\] where the bidegree of each component is given by
\[|\mub|=(-1,2), |\delb|=(0,1), |\del|=(1,0)\text{ and }|\mu|=(2,-1).\]
We then let 
 $\bar \delta^* = - \star \delta \star$ for $\delta= \mub, \delb, \del, \mu$
 and we have the bidegree decomposition \[d^*=\mub^*+\delb^*+\del^*+\mu^*.\]
 where
 \[|\mub^*|=(1,-2), |\delb^*|=(0,-1), |\del^*|=(-1,0)\text{ and }|\mu^*|=(-2,1).\]

Let $ L : \cA^{p,q} \to \cA^{p+1,q+1}$ be the real $(1,1)$-operator given by $L(\eta) = \omega \wedge \eta$. Let 
 $\Lambda = L^* = \star^{-1} L \star$. Then $\star \Lambda = L \star$ and $\star L = \Lambda \star$.
 Let $P^k = \Ke \Lambda \cap \cA^k$ denote the primitive forms of total degree $k$.

It  is well known that $\{L, \La, [L,\La] \}$ defines a representation of $sl(2,\CC)$ and induces the Lefschetz decomposition on forms:

\begin{Lemma} \label{Lefdecomp} We have
\[
\cA^k = \bigoplus_{r=0}^{k/2} L^r P^{k-2r},
\]
and this direct sum decomposition respects the $(p,q)$ bigrading.
\end{Lemma}

Let $[A,B] = AB-(-1)^{|A| |B|}BA$ be the graded commutator, where $|A|$ denotes the total degree of $A$. This defines a graded Poisson algebra
\[
[A, BC] =[A,B]C + (-1)^{|A| |B|} B[A,C]
\] 
 The following is well known (e.g. \cite{Huy} Corollary 1.2.28):
 
\begin{Lemma} \label{[L^j,Lambda]} 
For all $j \geq 0$  and $\alpha \in \cA^k$
\[
[L^j, \Lambda] \alpha = j(k-n+j-1) L^{j-1} \alpha.
\]
\end{Lemma}

By induction, and the fact that $[d,L]$ and $L$ commute, we have:

\begin{Lemma} \label{[d,L^n]}
For all $n \geq 1$
\[
[d,L^n ] = n[d,L] L^{n-1},
\]
and
\[
\star [d,L] \alpha = (-1)^{k+1} [d^*, \Lambda] \star  \alpha \quad \textrm{for  $ \alpha \in \cA^k$}.
\]

\end{Lemma}

Let $\II$ be the extension of $J$ to all forms as an algebra map with respect to wedge product, so that
 $\mathbb{I}_{p,q}$ acts on $\cA^{p,q}$ by multiplication by $i^{p-q}$. 
Then $\mathbb{I}_{p,q}^2 = (-1)^{p+q}$ so that $\mathbb{I}_{p,q}^{-1} = (-1)^{p+q} \mathbb{I}_{p,q}$. 
Note that $\mathbb{I}$ and $\star$ commute, and $\mathbb{I}$  and $L^n$ commute for all $n \geq 0$.
The following is a direct calculation.

\begin{Lemma} \label{conjbyI}
If an operator $T_{r,s}:\cA^{p,q} \to \cA^{p+r,q+s}$ has bidegree $(r,s)$, then 
\[
\mathbb{I}_{r+p,s+q}^{-1} \circ  T_{r,s} \circ  \mathbb{I}_{p,q} = (-i)^{r-s} T_{r,s}.
\]
\end{Lemma}

The above result readily implies that 
\[\mathbb{I}^{-1} \circ d \circ \mathbb{I}=-i(\mub-\delb+\del-\mu).\]
Finally, the following is well known (e.g.  \cite{Huy} Proposition 1.2.31):

 \begin{Lemma} \label{starL=LI}
 If $M$ is an almost Hermitian manifold of dimension $2n$, then for all $j \geq 0$ and all $\alpha \in P^k$,
\[
\star L^j  \alpha = (-1)^{\frac {k(k+1)}{2} } \frac{j!}{(n-k-j)!} L^{n-k-j} \, \II \alpha .
\]
\end{Lemma}

\section{Almost Hermitian identities}

By the previous section, any differential form $\eta$ can be written as $\eta = L^j \alpha$ for unique $j,k \geq 0$ and $\alpha \in P^k$. 
We now state the main result:

\begin{Theorem} \label{AHidentity}
 For any almost Hermitian manifold of dimension $2n$, let $\alpha \in P^k$,  with $d \alpha$ written as
 \begin{equation} \label{dalpha}
  d \alpha = \alpha_0 + L \alpha_1 + L^2 \alpha_2 +\cdots,
\end{equation}
 for  unique $\alpha_r \in P^{k+1-2r}$. Then, for all $j \geq 0$, 
  \begin{align*} 
[\Lambda, d] L^j \alpha - \star\, \mathbb{I}^{-1} \,d \, \mathbb{I}\,  \star  L^j \alpha &=   \frac{1}{j+1}  \,  \II^{-1} \,   [d^*, \Lambda]  \, \II   \,  L^{j+1} \alpha  \\
&+  j \Lambda [d,L]L^{j-1} \alpha + j(j-1)(k-n+j-1) [d,L] L^{j-2} \alpha \\
& + \sum_{r=2}^\infty   f_{n,j,k}(r) L^{j+r-1} \alpha_r,
\end{align*}
where
\[
f_{n,k,j}(r) = (r(n-k+r)-j) +(-1)^{r}  \frac{j!(n-k-j+r)!}{(j+r-1)!(n-k-j)!}.
\]
\end{Theorem}
 
\begin{Remark}
In the almost K\"ahler case we have $[d^*, \Lambda ]= [d,L]=0$, and $d \alpha = \alpha_0 + L \alpha_1$, so we recover the identity  
 \[
 [\Lambda, d] =  \star\, \mathbb{I}^{-1} \,d \, \mathbb{I}\,  \star,
  \]
 as expected.
\end{Remark}

\begin{proof} The proof consists of several calculations using the lemmas in the previous section. 
Using $[\II,L] = 0$, and $\II^2 = (-1)^k$ on $\cA^k$, we have
\begin{align*}
\star\, \mathbb{I}^{-1} \,d \, \mathbb{I}\,  \star  \eta &=  \star\, \mathbb{I}^{-1} \,d \, \mathbb{I}  \left(  
(-1)^{\frac {k(k+1)}{2} } \frac{j!}{(n-k-j)!} L^{n-k-j} \, \II \alpha \right) \\
&=  (-1)^{\frac {k(k+1)}{2} +k} \frac{j!}{(n-k-j)!} \, \star \II^{-1} \, d L^{n-k-j} \alpha.
\end{align*}
By Lemma \ref{[d,L^n]} this is equal to 
\begin{equation}\label{this}
(-1)^{\frac {k(k+1)}{2} +k} \frac{j!}{(n-k-j)!} \, \star \II^{-1} \,  L^{n-k-j} d   \alpha
\\
+ (-1)^{\frac {k(k+1)}{2} +k} \frac{j!}{(n-k-j-1)!} \, \star \II^{-1} \,  [d,L]  L^{n-k-j-1} \alpha.
\end{equation}
We first simplify each of these last two summands. 
By Equation (\ref{dalpha}), the fact that $\star$ commutes with $\II$, and Lemma \ref{starL=LI} applied to $\alpha_r \in P^{k+1-2r}$, 
the first summand of Equation (\ref{this}) is equal to:
\begin{multline*}
(-1)^{\frac {k(k+1)}{2} +k} \frac{j!}{(n-k-j)!} \, \star \II^{-1} \,  \left( \sum_{r=0}^\infty L^{n-k-j+r} \alpha_r \right) \\
 =
 (-1)^{\frac {k(k+1)}{2} +k} \frac{j!}{(n-k-j)!} \, \II^{-1} \,  \left(   \sum_{r=0}^\infty
(-1)^{\frac {(k+1-2r)(k-2r+2)}{2} } \frac{(n-k-j+r)!}{(j+r-1)!} L^{j+r-1} \,  \II \alpha_r \right) \\
=  \sum_{r=0}^\infty (-1)^{r+1}  \frac{j!(n-k-j+r)!}{(j+r-1)!(n-k-j)!} L^{j+r-1} \, \alpha_r .
\end{multline*}

For the second summand, we use the fact that for all $m \geq 0$ and  $\beta \in \cA^k$, 
\[
\star L^m [d,L] \beta = \star [d,L] L^m \beta  = (-1)^{k+1} [d^*,\Lambda]  \star L^m \beta.
\]
So, the second summand in Equation (\ref{this}) is equal to
\begin{multline*}
(-1)^{\frac {k(k+1)}{2} +k} \frac{j!}{(n-k-j-1)!} \, \star \II^{-1} \,   [d,L]  L^{n-k-j-1} \alpha
\\ =
(-1)^{\frac {k(k+1)}{2}+1} \frac{j!}{(n-k-j-1)!} \,  \II^{-1} \,   [d^*, \Lambda]  \star L^{n-k-j-1} \alpha
\\
= 
(-1)^{\frac {k(k+1)}{2}+1} \frac{j!}{(n-k-j-1)!} \,  \II^{-1} \,   [d^*, \Lambda]  (-1)^{\frac {k(k+1)}{2} } \frac{(n-k-j-1)!}{(j+1)!} L^{j+1} \, \II  \alpha
\\
=
  \frac{-1}{j+1}  \,  \II^{-1} \,   [d^*, \Lambda]  \, \II   \,  L^{j+1} \alpha,
\end{multline*}
where in the second to last step we used Lemma \ref{starL=LI}.

In summary, we have
\begin{equation} \label{LHS1}
\star\, \mathbb{I}^{-1} \,d \, \mathbb{I}\,  \star  \eta
=
 \sum_{r=0}^\infty (-1)^{r+1}  \frac{j!(n-k-j+r)!}{(j+r-1)!(n-k-j)!} L^{j+r-1} \, \alpha_r
- \frac{1}{j+1}  \,  \II^{-1} \,   [d^*, \Lambda]  \, \II   \,  L^{j+1} \alpha.
\end{equation}

We now compute $[\Lambda, d] \eta$, by first computing $\Lambda d L^j \alpha$, using that all $\alpha_r$ are primitive.
By Equation(\ref{dalpha}), Lemma \ref{[L^j,Lambda]}, and Lemma \ref{[d,L^n]}, we have:
\begin{align*}
\Lambda d L^j \alpha &= \Lambda L^j d \alpha + \Lambda [d,L^j] \alpha \\
& = \Lambda L^j \left( \sum_{r=0}^\infty L^r \alpha_r \right)  + j \Lambda [d,L]L^{j-1} \alpha \\
& =  \sum_{r=0}^\infty \Lambda L^{j+r} \alpha_r + j \Lambda [d,L]L^{j-1} \alpha \\
& = - \sum_{r=0}^\infty   (j+r)(k+1-2r-n+j+r-1) L^{j+r-1} \alpha_r + j \Lambda [d,L]L^{j-1} \alpha .
\end{align*} 
Next using, $\alpha$ is primitive, and Lemma \ref{[L^j,Lambda]}  again, we have 
\begin{align*}
d \Lambda L^j \alpha &= -j(k-n+j-1) d L^{j-1} \alpha \\
& = -j(k-n+j-1) L^{j-1} d \alpha  -j(k-n+j-1) (j-1) [d,L] L^{j-2} \alpha \\
& = -j(k-n+j-1)  \left( \sum_{r=0}^\infty L^{j+r-1} \alpha_r \right)   -j(k-n+j-1) (j-1) [d,L] L^{j-2}  \alpha.
\end{align*}
So, 
\begin{equation*} 
[\Lambda, d] \eta =   \sum_{r=0}^\infty   (r(n-k+r)-j) L^{j+r-1} \alpha_r 
+  j \Lambda [d,L]L^{j-1} \alpha + j(j-1)(k-n+j-1) [d,L] L^{j-2} \alpha.
\end{equation*}
Using this last equation and combining with Equation (\ref{LHS1}) we obtain the desired result:
\begin{align*} 
[\Lambda, d] \eta - \star\, \mathbb{I}^{-1} \,d \, \mathbb{I}\,  \star  \eta &=   \frac{1}{j+1}  \,  \II^{-1} \,   [d^*, \Lambda]  \, \II   \,  L^{j+1} \alpha  \\
&+  j \Lambda [d,L]L^{j-1} \alpha + j(j-1)(k-n+j-1) [d,L] L^{j-2} \alpha \\
& + \sum_{r=0}^\infty   f_{n,k,j}(r) L^{j+r-1} \alpha_r,
\end{align*}
where
\[
f_{n,k,j}(r) = (r(n-k+r)-j) +(-1)^{r}  \frac{j!(n-k-j+r)!}{(j+r-1)!(n-k-j)!}.
\]
It is a curious fact that $f(0)=f(1) = 0$, whereas for $r \geq 2$, $f(r)$ is in general non-zero.
\end{proof}

\section{Applications}

On an almost K\"{a}hler manifold, using the bidegree decompositions of $d$ and $d^*$, one may derive from (\ref{theequation}) the relation 
\[[\Lambda,\del]=i\delb^*,\]
involving $\Lambda$, $\del$ and the adjoint of $\delb$. 
For a non-K\"{a}hler Hermitian manifold there is an additional term 
\[[\Lambda,\del]=i(\delb^*+ \taub^*)\]
where $\taub=[\Lambda,[\delb,L]]$ is the zero-order \textit{torsion operator} (see \cite{De, SWH}). In the case of $(0,q)$-forms this gives
\[
 \Lambda \del \alpha = i \delb^* \alpha + i [\Lambda, \delb^*] L  \alpha.
\] 
Next we use Theorem \ref{AHidentity} to derive this local identity also in the non-integrable case.

\begin{Proposition}\label{0qforms}
For all $\alpha \in \cA^{0,q}$ in an almost Hermitian manifold we have 
\[
 \Lambda \del \alpha = i \delb^* \alpha + i [\Lambda, \delb^*] L  \alpha.
\] 
\end{Proposition}
\begin{proof}
 By bidegree reasons $\alpha$ is a primitive form and we have $d\alpha=\alpha_0+L\alpha_1+L^2\alpha_2$ where $\alpha_i$ are primitive. By expanding each term in the equality of Theorem \ref{AHidentity} with respect to the bidegree decomposition $d=\mub+\delb+\del+\mu$, in the case $j=0$, we obtain:
\[
 [\Lambda,d]\alpha=\Lambda d \alpha =   \Lambda (\del+\mu) \alpha,
\]
 \[
\star\, \mathbb{I}^{-1} \,d \, \mathbb{I}\,  \star  \alpha= i (\delb^*-\mub^*) \alpha,
\]
and
\[
 \mathbb{I}^{-1} \,[d^*, \Lambda ] \, \mathbb{I} \, L \alpha = i [\Lambda, \delb^*-\mub^*] L  \alpha.
\]
In particular, all terms decompose into sums of pure bidegrees $(0,q-1)$ and $(1,q-2)$.
Note as well that the remaining term 
\[f_{n,0,k}(2)L\alpha_2\]
given in Theorem \ref{AHidentity}
has pure bidegree $(1,q-2)$, since $\alpha_2$ must have bidegree $(0,q-3)$.
By putting together all terms of bidegree $(0,q-1)$ we obtain the desired identity.
\end{proof}

\begin{Remark}
 The proof of Proposition \ref{0qforms} gives a second identity relating the operators $\Lambda$, $\mu$ and $\mub$ and their adjoints, which also contains the term $f_{n,0,k}(2)L\alpha_2$.
 For forms in $\cA^{0,2}$, this extra term vanishes by bidegree reasons, since $\alpha_2=0$.
 Then the second identity reads
 \[
 \Lambda \mu \alpha = -i \mub^* \alpha - i [\Lambda, \mub^*] L  \alpha.
\]
This corrects the 
 identity
 \[
 [\Lambda, \mu] = -i \mub^*
\] 
known in the almost K\"{a}hler case for arbitrary forms (see \cite{AK}).
\end{Remark}

The previous proposition can be used to give a uniqueness result for the Dirichlet problem on compact domains with boundary.

\begin{Corollary}
Let $\Omega$ be a compact domain in an almost complex manifold $(M,J)$, with smooth boundary, and let $g: \Omega \to \CC$, and $\phi : \del \Omega \to \CC$ be smooth. Then the Dirichlet problem,
\[
\del \delb u = g \quad \quad \textrm{with} \quad u |_{\del \Omega} = \phi,
\]
has at most one solution $u: \Omega \to \CC$. 

In particular, if $(M,J)$ is a compact connected almost complex manifold, and $f: M \to \CC$ is a smooth map of almost complex manifolds,  then $f$ is constant.
\end{Corollary}

\begin{proof}
It suffices to show the only solution to the homogenous equation with $g=0$ is a constant function.

In any coordinate chart $\psi:V  \to \R^{2n}$ containing any maximum point, we pullback $J$ to 
$\psi(V)$ and consider the $J$-preserving map $u \circ \psi^{-1}: \psi(V) \to \CC$. The components of $d$ are natural with respect to this $J$-preserving map and we use the compatible metric on $\psi(V)$ to define $\Lambda$ and $\delb^*$. Then by Proposition \ref{0qforms} with $q=1$
we obtain
\[
-i \Lambda \del \delb u =  \delb^* \delb u  + [ \Lambda, \delb^*] L \delb u
\]
 on $\psi(V)$. Note $\delb^* \delb$ is quadratic, self-adjoint, and positive, and 
$[ \Lambda, \delb^*] L \delb$ is first order since $[ \Lambda, \delb^*]  = [d,L]^*$ is zeroth order, because $[d,L] \eta = d \omega \wedge \eta$. Then  the right hand side is zero, so the maximum principle due to E. Hopf applies \cite{Hopf}, showing $u$ is constant in a neighborhood of the maximum point and therefore, by connectedness, $u$ is constant.

The final claim follows taking $\Omega = M$, with empty boundary, $g=0$, and noting the condition that $f$ is a map of almost complex manifolds implies $\delb f = 0$.

\end{proof}

\begin{Remark}
 In \cite{CWDol}, we introduce a Dolbeault cohomology theory that is valid for all almost complex manifolds. The above corollary is key in showing that, for a compact connected almost complex manifold, this cohomology is well-behaved in lowest bidegree, in the sense that $H_{\Dol}^{0,0}(M)\cong \CC$.
\end{Remark}

Finally, we refer the reader to the work of Feehan and Leness \cite{FL}, where the relation of Proposition \ref{0qforms}, for $q=1$, is used to show that the moduli spaces of unitary anti-self-dual connections over any almost Hermitian 4-manifold is almost Hermitian, whenever the Nijenhuis tensor has sufficiently small $C^0$-norm.

\bibliographystyle{alpha}

\bibliography{biblio}

\end{document}